\documentclass[11pt,leqno]{article}
\usepackage{amsmath, amscd, amsthm, amssymb, graphics, xypic, mathrsfs, setspace, fancyhdr, times, bm, enumitem}
\xyoption{all}
\usepackage[colorlinks=true,pagebackref=true]{hyperref} 
\hypersetup{backref}

\newcommand{\Spec}{\operatorname{Spec}}

\newcommand{\sma}{{\scriptstyle{\wedge}}}

\newcommand{\R}{{\mathbb R}}
\newcommand{\C}{{\mathbb C}}

\newcommand{\Z}{{\mathbb Z}}
\newcommand{\N}{{\mathbb N}}
\newcommand{\A}{{\mathbb A}}

\newcommand{\aone}{{\mathbb A}^1}

\newcommand{\gm}[1]{{{\mathbf G}_{m}^{#1}}}

\newcommand{\ho}[1]{\mathscr{H}_{\aone}({#1})}

\newcommand{\bpi}{\bm{\pi}}
\newcommand{\piaone}{{\bpi}^{\aone}}

\newcommand{\KMW}[1]{{{\mathbf K_{{#1}}^{\mathrm{MW}}}}}

\newcommand{\Addresses}{{
 \bigskip
 \footnotesize
J.~Fasel, Institut Fourier - UMR 5582, Universit\'e Grenoble Alpes, CS 40700, 38058 Grenoble cedex 9, France; \textit{E-mail address:} \url{jean.fasel@gmail.com}

}}

\newcounter{intro}
\theoremstyle{plain}
\newtheorem{thm}{Theorem}[subsection]

\newtheorem{cor}[thm]{Corollary}
\newtheorem{prop}[thm]{Proposition}
\newtheorem*{claim*}{Claim}  

\newtheorem*{thm*}{Theorem}
\newtheorem*{problem*}{Problem}

\theoremstyle{definition}

\theoremstyle{remark}
\newtheorem{rem}[thm]{Remark}

\numberwithin{equation}{section}

\begin{document}
\pagestyle{fancy}
\renewcommand{\sectionmark}[1]{\markright{\thesection\ #1}}
\fancyhead{}
\fancyhead[LO,R]{\bfseries\footnotesize\thepage}
\fancyhead[LE]{\bfseries\footnotesize\rightmark}
\fancyhead[RO]{\bfseries\footnotesize\rightmark}
\chead[]{}
\cfoot[]{}
\setlength{\headheight}{1cm}

\author{Jean Fasel}

\title{{\bf The Vaserstein symbol on real smooth affine threefolds}}
\date{}
\maketitle

\begin{abstract}
We give a necessary and sufficient topological condition for the Vaserstein symbol to be injective on smooth affine real threefolds. 
\end{abstract}

\begin{footnotesize}
\tableofcontents
\end{footnotesize}

\section*{Introduction}

Let $R$ be a ring. Recall that a row $(a_1,\ldots,a_n)\in R^n$ is unimodular if the ideal generated by the $a_i$ is $R$ itself, i.e. $\langle a_1,\ldots,a_n\rangle=R$. It is equivalent to saying that the associated homomorphism $R^n\to R$ is surjective. The second description makes clear that the group $GL_n(R)$ acts on the set $Um_n(R)$ of unimodular rows (by precomposition) and one can consider the orbit set $Um_n(R)/GL_n(R)$. It follows that any subgroup of $GL_n(R)$ also acts on $Um_n(R)$, and in particular so does $E_n(R)\subset GL_n(R)$. It is often useful to see that set of orbits as pointed by the class of $e_1:=(1,0,\ldots,0)$. In \cite[\S 5]{Suslin76}, the authors defined the so-called \emph{Vaserstein symbol} $V:Um_3(R)/E_3(R)\to W_E(R)$, where the right-hand term is the elementary symplectic Witt group described in \cite[\S 3]{Suslin76} (see also \cite[\S 2]{Fasel09b}). The Witt group $W_E(R)$ is generated by skew-symmetric invertible matrices and the symbol $V$ can be explicitly described by 
\[
V(a_1,a_2,a_3)=\begin{pmatrix} 0 & -a_1 & -a_2 & -a_3 \\ a_1 & 0 & -b_3 & b_2 \\ a_2 & b_3 & 0 & -b_1 \\ a_3 & -b_2 & b_1 & 0\end{pmatrix}
\]
where $b_1,b_2,b_3$ are such that $\sum_{i=1}^3 a_ib_i=1$ (the symbol is independent of the choice of such elements). Further, they prove that the map $V$ is bijective under additional condition on $R$ (\cite[Theorem 5.2]{Suslin76}) which are satisfied if $R$ is of Krull dimension $2$. Consequently, the orbit set $Um_3(R)/E_3(R)$ inherits in that case an abelian group structure which was widely used. Later, R. Rao and W. van der Kallen proved in \cite[Corollary 3.5]{Rao94} that $V$ was also bijective when $R$ is a smooth $k$-algebra, where $k$ is a field which is perfect and $C_1$. In the same paper, they also proved that the symbol $V$ was in general not injective for threefolds. More precisely, they proved that $V$ was not injective when $R$ was the coordinate ring of the real algebraic sphere of dimension $3$. Later, D. R. Rao and N. Gupta showed that there was infinitely many real (singular) threefolds for which the Vaserstein symbol was not injective (\cite{Gupta14b}), a result later improved in \cite{Kolte16} where the authors exhibited an uncountable family of smooth real threefold for which $V$ fails to be injective. This activity culminated in a conjecture attributed to the author of the present paper (\cite[Introduction]{Kolte16}) which gave a necessary and sufficient topological condition making the Vaserstein symbol injective for smooth real affine threefolds. 

The aim of this note is to prove the conjecture. More precisely, we prove the following theorem.

\begin{thm*}
Let $X=\Spec (R)$ be a smooth affine real threefold. Let $\mathcal C$ be the set of compact connected components of $X(\R)$ (in the Euclidean topology). Then, the Vaserstein symbol
\[
Um_3(R)/E_3(R)\to W_E(R)
\]
is injective if and only if $\mathcal C=\emptyset$.
\end{thm*}

Our method is as follows. In \cite[proof of Theorem 4.3.1]{Asok14b}, we observed that the Vaserstein symbol had an interpretation in the realm of $\aone$-homotopy theory. More precisely, let $k$ be a perfect field. The smooth affine quadric $Q_5$ with $k[Q_5]=k[x_1,x_2,x_3,y_1,y_2,y_3]/\langle\sum x_iy_i-1\rangle$ is isomorphic to the quotient of algebraic varieties $SL_4/Sp_4$. The latter is in turn isomorphic to the the affine scheme $A^\prime_4$ representing the functor assigning to a ring $R$ the set of invertible skew-symmetric matrices with trivial Pfaffian. The composite $Q_5\to SL_4/Sp_4\to A^\prime_4$ associates to a $6$-tuple $(a_1,a_2,a_3,b_1,b_2,b_3)$ such that $\sum a_ib_i=1$ the matrix $V(a_1,a_2,a_3)$ described above. Now, there is a stabilization map $SL_4/Sp_4\to SL_6/Sp_6$ and it turns out that this map actually determines the injectivity of the Vaserstein symbol. Indeed, let $\ho k$ be the $\aone$-homotopy category defined by Morel and Voevodsky in \cite{Morel99}. We then have $\mathrm{Hom}_{\ho k}(X,Q_5)=Um_3(R)/E_3(R)$ for any smooth affine threefold $X=\Spec R$ and $\mathrm{Hom}_{\ho k}(X,SL_6/Sp_6)=W_E(R)$, while the stabilization map $Q_5\to SL_6/Sp_6$ precisely induces the Vaserstein symbol. In this context, we can use the computation of the homotopy sheaves of $Q_5\simeq \A^3\setminus 0$ obtained in \cite{Asok12c} to prove that the symbol $V$ is injective if $\mathcal C$ is empty. To prove that this condition is also necessary, we produce explicitly a morphism $\A^4\setminus 0\to Q_5$ whose composite with $Q_5\to SL_6/Sp_6$ is homotopy trivial and show that its real realization is the Hopf map $S^3_{\R}\to S^2_{\R}$. If $\mathcal C\neq\emptyset$, this allows to produce non trivial elements in $Um_3(R)/E_3(R)$ whose image under $V$ is trivial. 

Let us now explain the organization of the paper. In Section \ref{sec:preliminaries}, we first quickly review the notions of $\aone$-homotopy theory needed to understand the proof of the main result. In particular, we explain the reinterpretation of the Vaserstein symbol sketched above, and recall the explicit models of motivic spheres obtained in \cite{Asok14b}. We also take the opportunity to provide a few cohomological computations based on a recent preprint of J. Jacobson (\cite{Jacobson16}). We then proceed with the proof that the condition $\mathcal C=\emptyset$ is sufficient for $V$ to be injective. In Section \ref{sec:necessary}, we then show that it is necessary. 

\subsection*{Notation and conventions}

The rings considered in this paper are commutative and unital. If $X$ is a real variety, we denote by $X_{\C}$ the scheme $X_{\C}:=X\times_{\R}\Spec (\C)$.

\subsection*{Acknowledgments}

The author wishes to thank Ravi Rao for motivating him to write this paper. It is also a pleasure to thank Aravind Asok for explaining the morphism $Q_7\to Q_4$ appearing in Section \ref{sec:necessary}.  


\section{Preliminaries}\label{sec:preliminaries}

The purpose of this section is to explain a few basic features of the motivic homotopy category $\ho k$ constructed by F. Morel and V. Voevodsky. In this part of the paper, $k$ is an infinite perfect field.

\subsection{Motivic homotopy theory}

Recall from \cite{Morel99} that a space $\mathscr X$ is a (Nisnevich) sheaf of simplicial sets $\mathrm{Sm}_k\to \mathrm{SSets}$. Basic examples of spaces are smooth $k$-schemes seen as functors $\mathrm{Sm}_k\to \mathrm{Sets}\to \mathrm{SSets}$ and simplicial sets seen as constant sheaves of simplicial sets. These include the simplicial spheres $S^i$ for any $i\geq 0$. In the category of spaces, one can perform many useful constructions, among which the smash product of (pointed) spaces. For instance, if we write $\gm{}$ for the groups scheme representing the functor $X\mapsto O_X(X)^\times$, we can form smash products of $\gm{}$ (pointed by $1$) with itself to obtain spaces $(\gm{})^{\sma j}$ for any integer $j\in\N$, and we can perform smash products with $S^i$ to obtain spaces $S^{i,j}:=S^i\sma(\gm{})^{\sma j}$ for any $i,j\in \N$ that we refer to as \emph{motivic spheres}.

Any space $\mathscr X$ has stalks at the points of the Nisnevich topology, and a (pointed) morphism of spaces $f:\mathscr X\to \mathscr Y$ is said to be a (pointed) weak-equivalence if it induces a weak-equivalence of simplicial sets on stalks. One can put a model structure on $\mathrm{Spc}_k$ that allow to invert weak-equivalences in a good way, and the corresponding homotopy category is the \emph{simplicial homotopy category of smooth schemes}. To get the motivic homotopy category $\ho k$, one further formally invert the morphisms $\mathscr X\times \A^1\to \mathscr X$ induced by the projection $\aone\to \Spec k$ via Bousefield localization. For any two spaces $\mathscr X$ and $\mathscr Y$, we denote by $[\mathscr X,\mathscr Y]_{\aone}$ the set of morphisms between them in $\ho k$.

As noted above, the spaces $S^{i,j}$ are objects of $\ho k$, and it is not clear in general that they have "geometric" models in that category, i.e. that there exists smooth schemes which are isomorphic to $S^{i,j}$. While the answer is known to be negative in general, there are two important particular cases. For any $n\geq 1$, let $Q_{2n-1}$ be the smooth affine $k$-scheme whose ring of global sections is $k[Q_{2n-1}]=k[x_1,\ldots,x_n,y_1,\ldots,y_n]/\langle \sum x_iy_i-1\rangle$, and let $Q_{2n}$ be the smooth affine scheme whose global sections are $k[Q_{2n}]=k[x_1,\ldots,x_n,y_1,\ldots,y_n,z]/\langle \sum x_iy_i-z(1-z)\rangle$. In \cite[Theorem 2.2.5]{Asok14}, we proved that there are explicit isomorphisms $S^{n-1,n}\simeq Q_{2n-1}\simeq \A^n\setminus 0$ and $S^{n,n}\simeq Q_{2n}$ in $\ho k$ for any $n\geq 1$. Interestingly, these motivic spheres appear naturally in the context of unimodular rows. Indeed, there are natural bijections $[X,Q_{2n-1}]_{\aone}=[X,\A^n\setminus 0]_{\aone}=Um_{n}(R)/E_n(R)$ for any smooth affine $k$-scheme $X=\Spec R$ (\cite[\S 4.2]{Asok12b} and \cite[Theorem 4.2.1]{Asok15b}). 

In $\ho k$, one can define for any $n\in\N$ the homotopy sheaves $\piaone_n(\mathscr X)$ associated to a (pointed) space $\mathscr X$. These are Nisnevich sheaves on $\mathrm{Sm}_k$ which are extremely hard to compute in general in analogy with classical topology. A celebrated result of F. Morel states that the homotopy sheaves of $\A^n\setminus 0$ are of the following form (provided $n\geq 2$):
\[
\piaone_{i}(\A^n\setminus 0)=\begin{cases} 0 & \text{if $i<n-1$.} \\ \KMW n & \text{if $i=n-1$.}\end{cases}
\]
The sheaf $\KMW n$ is the \emph{(unramified) Milnor-Witt $K$-theory sheaf} and is defined in \cite[\S 3]{Morel08}. For $i\geq n$, our knowledge is still limited (see however \cite{Asok12a} for $n=2$ and \cite{Asok12c} for $n=3$).  

To conclude this section, let us recall from \cite[\S 3.3]{Morel99} that there is a realization functor $\ho k\to \mathcal H$ in case $k=\R$, where $\mathcal H$ is the usual homotopy category of topological spaces. Indeed, there is a realization functor $\ho k\to \mathcal H(\Z/2)$ where $\mathcal H(\Z/2)$ is the homotopy category of equivariant $\Z/2$-topological spaces. This functor associates to a smooth scheme $X$ over $\R$ the topological space $X(\C)$ together with the $\Z/2$-action induced by complex conjugation. Taking fixed points, we finally get a functor $\ho k\to \mathcal H$ associating to a smooth scheme $X$ the topological space $X(\R)$.

\subsection{Cohomology with coefficients in $\mathbf{I}^n$}

As mentioned in the previous section, the Milnor-Witt $K$-theory sheaf $\KMW n$ appears as the first nontrivial homotopy sheaf of $\A^n\setminus 0$. This sheaf is possibly better understood using two more classical sheaves. For the first one, assume that $k$ is of characteristic different from $2$ and consider the presheaf $X\mapsto W(X)$, where $W(X)$ is the Witt group of symmetric bilinear forms as defined by M. Knebusch (\cite{Knebusch76}). For any $n\in\N$, one can consider its subpresheaf $I^n(X)\subset W(X)$ of powers of the fundamental ideal $I(X)\subset W(X)$ and its associated (Zariski) sheaf $\mathbf{I}^n$ on $\mathrm{Sm}_k$. By construction, there is an epimorphism of sheaves $\KMW n\to \mathbf{I}^n$ for any $n\in \N$ (\cite[\S3]{Morel99}). For real varieties, the latter recovers very important topological information about $X(\R)$. Indeed, for any $m\in \N$, let $H^m(X(\R),\Z)$ be the singular cohomology group of $X(\R)$. Then, one has the following comparison theorem which is due to J. Jacobson (\cite[Corollary 8.3]{Jacobson16}).

\begin{thm}
Let $X$ be scheme over $\R$ which is separated and of finite type and let $d$ be its dimension. Then, the signature map induces isomorphisms
\[
H^m(X,\mathbf{I}^n)\to H^m(X(\R),\Z)
\]
for any $n\geq d+1$ and any integer $m\geq 0$. 
\end{thm}

\begin{cor}\label{cor:computation}
Let $X$ be a smooth real affine variety and let $\mathcal C$ be the set of compact connected components of $X(\R)$. Let $\mathcal C_1\subset \mathcal C$ be the subset of oriented such components and $\mathcal C_2=\mathcal C\setminus \mathcal C_1$. We have then for any $j>0$ an isomorphism
\[
H^d(X,\mathbf{I}^{d+j})\simeq \bigoplus_{\mathcal C_1}\Z  \oplus \bigoplus_{\mathcal C_2} \Z/2. 
\]
\end{cor}

\begin{proof}
In view of the above theorem, it suffices to show that $H^d(C,\Z)=\Z$ if $C\in \mathcal C_1$, $H^d(C,\Z)=\Z/2$ if $C\in \mathcal C_2$ and $H^d(C,\Z)=0$ if $C$ is a non compact connected component of $X(\R)$. The first case is a direct consequence of Poincar\'e duality \cite[Theorem 3.30]{Hatcher02} and the second case follows from the universal coefficient theorem and \cite[Theorem 3.26, Corollary 3.28]{Hatcher02}. Finally, suppose that $C$ is a non compact connected component of $X(\R)$. Arguing as in \cite[Lemma 4.16]{Bhatwadekar06}, we can see $C$ as a strict submanifold of a compact connected component $C^\prime$. It follows easily from the two previous cases that $H^d(C,\Z)=0$.
\end{proof}

Let now $\mathbf{K}^{\mathrm{M}}_{n}$ be the unramified Milnor $K$-theory sheaf considered for instance in \cite[\S 2]{Morel05}. Its sections over a finitely generated field extension $L/k$ form precisely the Milnor $K$-theory group $K^M_n(L)$. One can consider its subsheaf $2\mathbf{K}^{\mathrm{M}}_{n}$ and it turns out (using Voevodsky's affirmation of Milnor conjecture) that there is an exact sequence of sheaves
\[
0\to 2\mathbf{K}^{\mathrm{M}}_{n}\to \KMW {n}\to \mathbf{I}^{n}\to 0.
\]
Using this sequence, we now slightly adapt \cite[Proposition 5.5]{Fasel08b} to get the following computation of $H^d(X,\KMW {d+j})$.

\begin{prop}\label{prop:refinement}
Let $X$ be a smooth real affine variety of dimension $d$. Then, for any $j>0$, we have an exact sequence
\[
H^d(X_{\C},\KMW {d+j})\to H^d(X,\KMW {d+j})\to H^d(X,\mathbf{I}^{d+j})\to 0.
\]
\end{prop}

\begin{proof}
The finite morphism $X_{\C}\to X$ induces a morphism between the explicit flasque resolutions of $\mathbf{K}^{\mathrm{M}}_{d+j}$ provided for instance in \cite[\S 6]{Rost96}. Now, \cite[proof Proposition 5.4]{Fasel08b} shows that this morphism of complexes yields a surjective homomorphism $H^d(X_{\C},\mathbf{K}^{\mathrm{M}}_{d+j})\to H^d(X,2\mathbf{K}^{\mathrm{M}}_{d+j})$. We conclude using the long exact sequence induced by the exact sequence of sheaves 
\[
0\to 2\mathbf{K}^{\mathrm{M}}_{d+j}\to \KMW {d+j}\to \mathbf{I}^{d+j}\to 0.
\]
\end{proof}

\subsection{Reinterpretation of the Vaserstein symbol}

We now elaborate a bit more the reinterpretation of the Vaserstein symbol in $\ho k$ hinted at in the introduction. As explained in \cite[proof of Theorem 4.3.1]{Asok14b}, we have an isomorphism $Q_5\to SL_4/Sp_4$ and we can consider the composite of this isomorphism with the stabilization morphism $SL_4/Sp_4\to SL/Sp$. This yields a map $Q_5\to SL/Sp$ which is called \emph{degree map} in {\it{loc.cit.}} The interesting feature of $SL/Sp$ is that it represents the (reduced) higher Grothendieck-Witt group $GW_1^3(X)$ as defined by Schlichting (see for instance \cite{Schlichting09} and \cite{Schlichting10}).  Now, this group coincides for affine varieties with the group $W_E(X)$ (\cite[\S 4]{Fasel10b}) and we get for any smooth affine variety $X=\Spec R$ over $k$ an induced map
\[
Um_3(R)/E_3(R)=[X,Q_5]_{\aone}\to [X,SL/Sp]_{\aone}=W_E(R)
\]
which coincides with the (opposite of the) Vaserstein symbol (\cite[proof of Theorem 4.3.1]{Asok14b}). On the other hand, it follows from \cite[Proposition 4.2.2]{Asok12c} that for a smooth affine threefold as above, we have $[X,SL/Sp]_{\aone}=[X,SL_6/Sp_6]_{\aone}$ and a fiber sequence
\begin{equation}\label{eqn:fiber}
\Omega (\A^5\setminus 0)\to Q_5\to SL_6/Sp_6
\end{equation}
It follows that the term $\Omega (\A^5\setminus 0)$ controls the kernel of the Vaserstein symbol.


\section{A sufficient condition for the symbol to be injective}\label{sec:sufficient}
In this section, we will show that if $\mathcal C$ is empty, then the Vaserstein symbol 
\[
V:Um_3(R)/E_3(R)\to W_E(X)
\]
is injective. In view of the $\aone$-homotopy reinterpretation above, it suffices then to show that for a real smooth affine threefold $X$ such that $\mathcal C=\emptyset$ we have a bijection $[X,\A^3\setminus 0]_{\aone}\to [X,\A^3,SL_6/Sp_6]_{\aone}$. To see this, we can consider the Postnikov towers (as constructed for instance in \cite[\S 6]{Asok12a}) associated to both $\mathscr X:=\A^3\setminus 0$ and $\mathscr Y:=SL_6/Sp_6$. This yields a commutative diagram of spaces
\[
\xymatrix@C=2em{K(\piaone_2(\mathscr X),1)\ar[r]\ar[d]_-{K(f,1)} & K(\piaone_3(\mathscr X),3)\ar[r]\ar[d]_-{K(f,3)} & \mathscr X^{(3)}\ar[r]\ar[d] & K(\piaone_2(\mathscr X),2)\ar[r]\ar[d]_-{K(f,2)} & K(\piaone_3(\mathscr X),4)\ar[d]_-{K(f,4)} \\
K(\piaone_2(\mathscr Y),1)\ar[r] & K(\piaone_3(\mathscr Y),3)\ar[r] & \mathscr Y^{(3)}\ar[r] & K(\piaone_2(\mathscr Y),2)\ar[r] & K(\piaone_3(\mathscr Y),4)}
\]
where the vertical maps are all induced by $f$. Since $X$ is a threefold, it follows from \cite[Proposition 6.2]{Asok12a} that $[X,\mathscr X^{(3)}]_{\aone}=[X,\mathscr X]_{\aone}$ (respectively that $[X,\mathscr Y^{(3)}]_{\aone}=[X,\mathscr Y]_{\aone}$). A diagram chase shows that it suffices to prove that the maps $K(f,i)$ induce isomorphisms after applying $[X,\_]_{\aone}$. Now, we know that $f$ induces an isomorphism $\piaone_2(\mathscr X)\to \piaone_2(\mathscr Y)$ and therefore $K(f,2)$ and $K(f,1)$ are isomorphisms. On the other hand, $[X,K(\mathcal F,4)]_{\aone}=H^4_{Nis}(X,\mathcal F)=0$ for any strictly $\aone$-invariant sheaf $\mathcal F$ on $\mathrm{Sm}_k$ thus showing that $K(f,4)$ also induces an isomorphism. We are then reduced to prove that $K(f,3)$ yields an isomorphism. The fiber sequence (\ref{eqn:fiber}) yields an exact sequence of sheaves
\[
\piaone_4(SL_6/Sp_6)\to \KMW 5\to \piaone_3(\A^3\setminus 0)\to \piaone_3(SL_6/Sp_6)\to 0.
\]
In order to prove that the morphism of sheaves $\piaone_3(\A^3\setminus 0)\to \piaone_3(SL_6/Sp_6)$ induces an isomorphism $H^3(X,\piaone_3(\A^3\setminus 0))\simeq H^3(X,\piaone_3(SL_6/Sp_6))$, it suffices to show that the morphism of sheaves $\piaone_4(SL_6/Sp_6)\to \KMW 5$ induces a surjective homomorphism $H^3(X,\piaone_4(SL_6/Sp_6))\to H^3(X,\KMW 5)$. Now, \cite[Proposition 4.2.2]{Asok12c} shows that $\piaone_4(SL_6/Sp_6)=\mathbf{GW}_5^3$ and that $\mathbf{GW}_5^3=\piaone_4(SL_6/Sp_6)\to \KMW 5$ is the morphism $\chi_5$ considered in \cite[proof of Theorem 4.3.1]{Asok12c}. Consider the finite morphism $X_{\C}\to X$. The transfer map for this morphism yields a commutative diagram
\[
\xymatrix{H^3(X_{\C},\mathbf{GW}_5^3)\ar[r]\ar[d] & H^3(X,\mathbf{GW}_5^3)\ar[d] \\
H^3(X_{\C},\KMW 5)\ar[r] & H^3(X,\KMW 5)}
\]
and Proposition \ref{prop:refinement} shows that the bottom horizontal map is surjective. It suffices then to prove that the left-hand vertical map is surjective to conclude. We know that the cokernel of $\mathbf{GW}_5^3\to \KMW 5$ is isomorphic to the sheaf denoted by $\mathbf{F}_5$ in \cite{Asok12c}, which is a quotient of the sheaf $\mathbf{T}_5$ considered in \cite[\S 3.6]{Asok12b}. As $\mathbf{I}^5=0$ on $X_{\C}$, The sheaf $\mathbf{T}_5$ reduces to the sheaf $\mathbf{S}_5$ which is a quotient of $\mathbf{K}^M_5/24$. We are then reduced to prove that $H^3(X_{\C},\mathbf{K}^M_5/24)=0$ which follows from the fact that $\mathbf{K}_2^M(\C)$ is divisible.

\begin{rem}
When $X_{\C}$ is moreover supposed to be rational and oriented, then $H^3(X,\KMW 5)=0$ by \cite[Theorem 5.7]{Fasel08b} and the argument is much simpler.
\end{rem}


\section{A necessary condition for the symbol to be surjective}\label{sec:necessary}

We now suppose that the set $\mathcal C$ is non empty and derive from this that the Vaserstein symbol is not injective. Our strategy is to construct an explicit morphism $Q_7\to \A^3\setminus 0$ whose composite with $\A^3\setminus 0\to SL_6/Sp_6$ is trivial. It follows that the image of $[X,Q_7]_{\aone}$ in $[X,\A^3\setminus 0]_{\aone}$ under this map furnishes candidates for the non injectivity of the Vaserstein symbol. We are reduced to show that under our condition, this image is non trivial.

We start with the definition of an explicit morphism $Q_7\to Q_4$. By definition, if $R$ is a commutative ring we have $Q_7(R)=\{a_1,\ldots,a_4,b_1,\ldots,b_4\in R\vert \sum a_ib_i=1\}$ and $Q_4(R)=\{x_1,x_2,y_1,y_2,z\in R\vert x_1y_1+x_2y_2=z(1-z)\}$. The set $Q_7(R)$ can be seen as the set $\{ M,N\in M_2(R) \vert \det M+\det N=1\}$ via 
\[
(a_1,\ldots,a_4,b_1,\ldots,b_4)\mapsto \left ( \begin{pmatrix} a_1 & a_2 \\ -b_2 & b_1\end{pmatrix}, \begin{pmatrix} a_3 & a_4 \\ -b_4 & b_3 \end{pmatrix} \right )
\]
and we define the morphism $f:Q_7\to Q_4$ via $(M,N)\mapsto (MN,\det M)$. As $\det(MN)=\det M\det N=(\det M)(1-\det M)$, we see that the 5-tuple $(MN,\det M)$ defines an element of $Q_4(R)$. Explicitly, the morphism is given by 
\[
f(a_1,\ldots,a_4,b_1,\ldots,b_4)=(a_1a_3-a_2b_4,a_1a_4+a_2b_3,-a_4b_2+b_1b_3,a_3b_2+b_1b_4,a_1b_1+a_2b_2).
\]
Let now $[-1]:S^0\to \gm{}$ be the morphism sending the non basepoint to $-1$. Smashing with any space $\mathscr X$, we get a morphism $\mathscr X\to \mathscr X\wedge \gm{}$ that we still denote by $[-1]$. By construction, the following diagram commutes
\[
\xymatrix{Q_7\ar[r]^-{f}\ar[d]_-{[-1]} & Q_4\ar[d]^-{[-1]} \\
Q_7\sma \gm{}\ar[r]_-{f\sma Id} & Q_4\sma \gm{}.}
\]
We next define a morphism 
\[
g:Q_4\sma \gm{}\to \A^3\setminus 0
\] 
via $(x_1,x_2,y_1,y_2,z,\alpha)\mapsto (2x_1,2x_2,(\alpha-1)z+1)$ and finally get a morphism
\[
\mathcal H:Q_7\stackrel{f}\to Q_4\stackrel{[-1]}\to Q_4\sma \gm{}\stackrel{g}\to \A^3\setminus 0.
\]

We know that $Q_7\simeq \A^4\setminus 0\simeq S^{3,4}$ in $\ho {\R}$ and it follows from \cite[\S 3.3]{Morel99} that $Q_7(\R)$ is homotopy equivalent to the real sphere $S_{\R}^3$. On the other hand, $\A^3\setminus 0\simeq S^{2,3}$ and its real realization is thus $S^2_{\R}$. The morphism $\mathcal H$ constructed above realizes in a morphism $S^3_{\R}\to S^2_{\R}$ and is therefore a multiple of the Hopf map. The next result shows that it is indeed the Hopf map.

\begin{prop}
The morphism $\mathcal H$ realizes (up to sign) in the Hopf map $S^3_{\R}\to S^2_{\R}$.
\end{prop}

\begin{proof}
We start by exhibiting an explicit homotopy equivalences $S^3_{\R}\to Q_7(\R)$.  Concretely, let $h:S^3_{\R}\to Q_7(\R)$ be given by 
\[
(x_1,x_2,x_3,x_4)\mapsto (x_1,x_2,x_3,x_4,x_1,x_2,x_3,x_4)
\]
The composite $S^3_{\R}\stackrel {h}{\to} Q_7(\R)\to \R^4\setminus 0\to S^3_{\R}$, where the second map is the projection onto the first four factors and the third map is the usual homotopy equivalence, is easily checked to be the identity. As both the second and third maps are homotopy equivalences, so is the first. Now, an easy computation yields that the composite
\[
S^3_{\R}\stackrel{f_1}\to Q_7(\R)\stackrel{\mathcal H}\to Q_4(\R)\stackrel{[-1]}\to (Q_4\sma \gm{})(\R)\stackrel {g}\to \R^3\setminus 0
\]
maps $(x_1,x_2,x_3,x_4)$ to $(2x_1x_3-2x_2x_4,2x_1x_4+2x_2x_3,x_3^2+x_4^2-x_1^2-x_2^2)$. This is the Hopf map made explicit in \cite[\S 4.5]{Rao94}.
\end{proof}

We now check that the composite
\[
Q_7\stackrel{\mathcal H}\to \A^3\setminus 0\to SL_6/Sp_6
\]
is homotopic to the point. Using the commutative diagram
\[
\xymatrix{Q_7\ar[r]^-{f}\ar[d]_-{[-1]} & Q_4\ar[d]^-{[-1]} \\
Q_7\sma \gm{}\ar[r]_-{f\sma Id} & Q_4\sma \gm{}}
\]
we see that it suffices to show that any morphism $Q_7\sma \gm{}\to SL_6/Sp_6$ is homotopy trivial. Now we have $[Q_7\sma \gm{},SL_6/Sp_6]_{\aone}=[S^{3,5},SL/Sp]=(\mathbf{GW}_4^3)_{-5}=\mathbf{GW}_{-1}^{-2}=0$, proving the claim. Consequently, the composite
\[
[X,\A^4\setminus 0]_{\aone}\to [X,\A^3\setminus 0]_{\aone}\to W_E(X)
\]
is trivial and it remains to show that the left-hand morphism is non trivial if $\mathcal C\neq \emptyset$. To this end, recall that realization induces maps $[X,\A^4\setminus 0]_{\aone}\to [X(\R),S^3_{\R}]$ and $[X,\A^3\setminus 0]_{\aone}\to [X(\R),S^2_{\R}]$ and that the morphism $\A^4\setminus 0\to \A^3\setminus 0$ realizes into the Hopf map. In classical topology, we have the Hopf fibration
\[
S^1_{\R}\to S^3_{\R}\to S^2_{\R}
\]
As $\pi_1(S^3_{\R})=0$, it follows that the map $S^1_{\R}\to S^3_{\R}$ is homotopy trivial and therefore for any $CW$-complex $Y$ the map $[Y,S^3_{\R}]\to [Y,S^2_{\R}]$ has trivial fiber. Suppose then that we have a class in $[X,\A^4\setminus 0]_{\aone}$ realizing into a non trivial element of $[X(\R),S^3_{\R}]$. It follows that its image in $[X,\A^3\setminus 0]_{\aone}$ is non trivial and therefore that the Vaserstein symbol is not injective. Now, $[X,\A^4\setminus 0]_{\aone}\simeq H^3(X,\KMW 4)$ which surjects onto $H^3(X,I^4)=H^3(X(\R),\Z)$ and we conclude using Corollary \ref{cor:computation}.

\begin{rem}
Let $X=\Spec R$ be a smooth real affine threefold with $\mathcal C\neq \emptyset$.
Using the identification $[X,\A^n\setminus 0 ]_{\aone}=Um_n(R)/E_n(R)$, the composite
\[
[X,\A^4\setminus 0]_{\aone}\to [X,\A^3\setminus 0]_{\aone}\to W_E(X)
\]
can be reinterpreted as a complex
\[
Um_4(R)/E_4(R)\stackrel{\alpha}\to Um_3(R)/E_3(R)\stackrel V\to W_E(R).
\]
where $\alpha([a_1,a_2,a_3,a_4])=[2a_1a_3-2a_2a_4,2a_1a_4+2a_2a_3,a_3^2+a_4^2-a_1^2-a_2^2]$ for any orbit $[a_1,a_2,a_3,a_4]\in Um_4(R)/E_4(R)$. The above proof shows that any non trivial element of $Um_4(R)/E_4(R)$ (which exist in case $\mathcal C\neq \emptyset$) will have non trivial image in $Um_3(R)/E_3(R)$ under $\alpha$ providing a non trivial element in the kernel of $V$.
\end{rem}

\begin{footnotesize}
\bibliographystyle{alpha}
\bibliography{Vstein}
\end{footnotesize}
\Addresses
\end{document}